\documentclass[12pt,oneside,reqno]{amsart}
\usepackage{mathrsfs}
\usepackage{graphics}
\usepackage{enumerate, amssymb}
\newcommand{\dif}{\mathrm{d}}

\newcommand{\be}{\begin{eqnarray}}
\newcommand{\ee}{\end{eqnarray}}
\newcommand{\ce}{\begin{eqnarray*}}
\newcommand{\de}{\end{eqnarray*}}
\newtheorem{theorem}{Theorem}[section]
\newtheorem{lemma}[theorem]{Lemma}
\newtheorem{remark}[theorem]{Remark}
\newtheorem{definition}[theorem]{Definition}
\newtheorem{proposition}[theorem]{Proposition}
\newtheorem{Example}[theorem]{Example}
\newtheorem{corollary}[theorem]{Corollary}
\def\e{\varepsilon}

\def\[{{\Big[}}
\def\]{{\Big]}}
\def\<{{\langle}}
\def\>{{\rangle}}
\def\({{\Big(}}
\def\){{\Big)}}

\def\no{\nonumber}
\def\bt{\begin{theorem}}
\def\et{\end{theorem}}
\def\bl{\begin{lemma}}
\def\el{\end{lemma}}
\def\br{\begin{remark}}
\def\er{\end{remark}}
\def\bx{\begin{Example}}
\def\ex{\end{Example}}
\def\bd{\begin{definition}}
\def\ed{\end{definition}}
\def\bp{\begin{proposition}}
\def\ep{\end{proposition}}
\def\bc{\begin{corollary}}
\def\ec{\end{corollary}}
\def\cA{{\mathcal A}}

\def\cD{{\mathcal D}}

\def\cH{{\mathcal H}}

\def\cL{{\mathcal L}}

\def\cP{{\mathcal P}}

\def\mE{{\mathbb E}}

\def\mH{{\mathbb H}}

\def\mN{{\mathbb N}}

\def\mP{{\mathbb P}}

\def\mR{{\mathbb R}}

\def\mU{{\mathbb U}}

\def\geq{\geqslant}
\def\leq{\leqslant}

\begin{document}

\allowdisplaybreaks

\title{On the path-independence of the Girsanov transformation for stochastic
evolution equations with jumps in Hilbert spaces*}

\author{Huijie Qiao$^1$ and Jianglun Wu$^2$}



\thanks{*This work was partly supported by NSF of China (No. 11001051, 11371352).}


\date{}

\dedicatory{1. School of Mathematics,
Southeast University\\
Nanjing, Jiangsu 211189,  China\\
hjqiaogean@seu.edu.cn\\
2. Department of Mathematics, Swansea University\\
Singleton Park, Swansea SA28PP, UK\\
j.l.wu@swansea.ac.uk}

\begin{abstract}
Based on a recent result in \cite{qw}, in this paper, we extend it to stochastic evolution equations
with jumps in Hilbert spaces. This is done via Galerkin type finite-dimensional approximations of the infinite-dimensional
stochastic evolution equations with jumps in a manner that one could then link the characterisation of the path-independence
for finite-dimensional jump type SDEs to that for the infinite-dimensional settings. Our result provides an intrinsic link of
infinite-dimensional stochastic evolution equations with jumps to infinite-dimensional (nonlinear) integro-differential equations.
\end{abstract}

\maketitle \rm

\section{Introduction}

The object of this paper is to characterise the path-independent property of the density process (via
the exponent process) of Girsanov transformation for stochastic evolution equations with jumps in
Hilbert spaces, a class of semi-linear stochastic partial differential equations (SPDEs) with jumps.
The latter class of SPDEs with jumps was studied analytically in \cite{pz}, which is continuously to be
a hot topic nowadays. As a result, we derive an intrinsic link of stochastic evolution equations with
jumps in Hilbert spaces to infinite-dimensional (nonlinear) integro-differential equations.
The derived nonlinear equations involve a Burgers-KPZ type nonlinearity, which should link
very well to statistical physics, such as studies of infinite interacting systems with
non-Gaussian noise driven stochastic dynamics (cf. e.g. the discussions in \cite{wawu}).

The investigation of path-indepedent property of the density of Girsanov transformation for It\^o type SDEs
on $\mR^d$ started in \cite{twwy,wy}, which was inspired by interesting considerations needed in
economics and mathematical finance (cf. e.g. references in \cite{wy}). In \cite{wawu}, Wang and the second author considered stochastic evolution equations in Hilbert spaces driven by cylindrical Brownian motion, and obtained a characterisation theorem via Galerkin type finite-dimensional approximations developed in \cite{FYWang}. Moreover, in \cite{qw}, we studied the characterising path-independence problem for
non-Lipschnitz SDEs with jumps on $\mR^d$, where we derived a link of SDEs with jumps to integro-differential equations via a proper setting of Girsanov transformation for SDEs with jumps.

Due to the complexity of SPDEs with jumps, our extension can not be a straightforward analogy to \cite{wawu}.
In fact, we have to overcome several difficulties arising from handling equations with jumps in infinite dimensions,
as well as need to establish a suitable Girsanov transformation for such equations and most importantly the
It\^o formula for the solutions of our stochastic evolution equations with jumps in Hilbert spaces, which we could
not find in the literature.  The obtained results extend both characterisation results in \cite{wawu} and in \cite{qw}.

The rest of the paper is organized as follows. In the next section, we will formulate and prove a proper Girsanov
theorem. The equations which we are concerned with are introduced in Section 3, where we follow \cite{FYWang} to develop
a  Galerkin type finite-dimensional approximation which extends the corresponding result in \cite{FYWang}.
Moreover, we derive an It\^o's formula for the solutions by utilising the Galerkin type finite-dimensional approximation.
Section 4 is devoted to proving the main results of the characterisation for infinite-dimensional equations
with jumps involving and not involving cylindrical Brownian motion, respectively. 

\section{The Girsanov theorem}\label{sgir}

In the section, we state and show a Girsanov theorem for Brownian motions and random measures in a real
separable Hilbert space. We introduce our framework first.

Give a filtered probability space $(\Omega, \mathscr{F}, \{\mathscr{F}_t\}_{t\in[0,\infty)},\mP)$. Let
$\{\beta_i(t,\omega)\}_{i\geq1}$ be a family of mutual independent one-dimensional Brownian motions on
$(\Omega, \mathscr{F}, \{\mathscr{F}_t\}_{t\in[0,\infty)},\mP)$. For a real separable Hilbert space
$(\mH,\langle\cdot,\cdot\rangle_{\mH},\|\cdot\|_{\mH})$, construct a cylindrical Brownian motion on $\mH$ with respect to $(\Omega, \mathscr{F}, \{\mathscr{F}_t\}_{t\in[0,\infty)},\mP)$ by
\begin{equation*}
W_t:=W_t(\omega):=\sum_{i=1}^\infty\beta_i(t,\omega)e_i, \quad \omega\in\Omega,\,\, t\in[0,\infty),
\end{equation*}
where $\{e_i\}_{i\geq1}$ is a complete orthonormal basis for $\mH$ which will be specified later.  It is easy to justify
that the covariance operator of the cylindrical Brownian motion $W$ is the identity operator $I$ on $\mH$. Note that $W$ is not a process on $\mH$. It is convenient to realize $W$ as a continuous process on an enlarged Hilbert space $\tilde{\mH}$, the completion of $\mH$ under the inner product 
$$
\<x,y\>_{\tilde{\mH}}:=\sum\limits_{i=1}^\infty 2^{-i}\<x,e_i\>\<y,e_i\>, \quad x, y\in\mH.
$$

Next, we introduce the jump measures. Let $(\mU,\mathscr{U},\nu)$ be a given $\sigma$-finite measure space (which is
interpreted as a parameter space measuring jumps) and let $\lambda:[0,\infty)\times\mU\to(0,1)$ be a measurable
function. Then, following e.g. Theorem I.8.1 of \cite{iw}, there exists an integer-valued random measure on $[0,\infty)\times\mU$
$$N_{\lambda}:\mathscr{B}([0,\infty)\times\mathscr{U}\times\Omega\to\mathbb{N}_0:=
\mathbb{N}\cup\{0\}\cup\{\infty\}$$
with intensity measure (i.e., its predictable compensator) $\lambda(t,u)\dif t\nu(\dif u)$:
$$\mE N_{\lambda}(\dif t, \dif u,\cdot)=\lambda(t,u)\dif t\nu(\dif u).$$
Set
$$\tilde{N}_\lambda(\dif t,\dif u):=N_\lambda(\dif t,\dif u)-\lambda(t,u)\dif t\nu(\dif u),$$
and then $\tilde{N}_\lambda(\dif t,\dif u)$  is the associated compensated martingale measure of $N_{\lambda}(\dif t, \dif u)$. Moreover,
we assume that $W_t$ and $ N_{\lambda}$ are independent.

Fix arbitrarily $T>0$ and $\mU_0\in\mathscr{U}$ with $\nu(\mU\setminus\mU_0)<\infty$. Set for any $t\in[0,T]$
\ce
Z_t&:=&W_t+\int_0^t\gamma(s,x)\dif s+\int_0^t\int_{\mU_0}\varphi(s,x,u)\tilde{N}_\lambda(\dif s,\dif u) \\
&&+\int_0^t\int_{\mU\setminus\mU_0}\psi(s,x,u) N_\lambda(\dif s,\dif u),\\
\de
where $\gamma: [0,T]\times\Omega\times\mH\mapsto\tilde{\mH}$ is $\cP\otimes\mathscr{B}(\mH)/\mathscr{B}(\tilde{\mH})$-measurable, and
$\varphi: [0,T]\times\Omega\times\mH\times\mU_0\mapsto\tilde{\mH}$ is $\cP\otimes\mathscr{B}(\mH)\otimes\mathscr{U}|_{\mU_0}/\mathscr{B}(\tilde{\mH})$-measurable
and $\psi: [0,T]\times\Omega\times\mH\times(\mU\setminus\mU_0)\mapsto\tilde{\mH}$ is $\cP\otimes\mathscr{B}(\mH)\otimes\mathscr{U}|_{\mU\setminus\mU_0}/\mathscr{B}(\tilde{\mH})$-measurable, therein $\cP$ stands for the predictable $\sigma$-algebra
on $[0,T]\times\Omega$. Put
\ce
\Lambda_t:&=&\exp\bigg\{-\int_0^t\<\gamma(s,x),\dif W_s\>_{\tilde{\mH}}-\frac{1}{2}\int_0^t
\left\|\gamma(s,x)\right\|_{\tilde{\mH}}^2\dif s-\int_0^t\int_{\mU_0}\log\lambda(s,u)N_{\lambda}(\dif s, \dif u)\\
&&\qquad\qquad\qquad -\int_0^t\int_{\mU_0}(1-\lambda(s,u))\nu(\dif u)\dif s\bigg\}.
\de
We will use $\Lambda_t$ to define a new probability measure $\hat{\mP}$ and show that under the measure $\hat{\mP}$,
$Z_t$ has a simpler form, namely, we get the following Girsanov theorem in our present framework.

\bt\label{tgir}
Assume that
\be\label{exma}
\mE[\Lambda_T]=1.
\ee
Then under the probability $\dif\hat{\mP}:=\Lambda_T\dif\mP$, the process $Z_t,t\in[0,T]$, has the following form
\ce
Z_t=\hat{W}_t+\int_0^t\int_{\mU_0}\varphi(s,x,u)\tilde{N}(\dif s,\dif u)
+\int_0^t\int_{\mU\setminus\mU_0}\psi(s,x,u) N_{\lambda}(\dif s,\dif u),\quad t\in[0,T],
\de
where, on the new filtered probability space $(\Omega,\mathscr{F},\{\mathscr{F}_t\}_{t\in[0,T]},\hat{\mP})$,
$\hat{W}_t:=W_t+\int_0^t\gamma(s,x)\dif s$ is a cylindrical Brownian motion, and the predictable
compensator and the compensated martingale measure of $N_{\lambda}(\dif t, \dif u)$ are $\dif t\nu(\dif u)$
and $\tilde{N}(\dif t, \dif u)$, respectively.
\et
\begin{proof}
For the cylindrical Brownian motion $W$, one could use the method similar to that in \cite[Theorem 10.14]{dpz} with
some slight modifications. The proof is then completed by directly applying Theorem 3.17 in \cite{jjas} to the random
measure $N_{\lambda}(\dif t, \dif u)$.
\end{proof}
In Section \ref{chthh}, the above theorem will be used to transform certain relevant processes.

\medspace

Next, we would like to present a sufficient condition on $\gamma$ and $\lambda$ such that $\Lambda_T$ fulfills the
assumption (\ref{exma}). Note that $\Lambda_t, t\in[0,T]$, is the Dol\'eans-Dade exponential of $M_t,t\in[0,T]$, i.e.,
\ce
M_t:&=&-\int_0^t\<\gamma(s,x),\dif W_s\>_{\tilde{\mH}}
+\int_0^t\int_{\mU_0}\frac{1-\lambda(s,u)}{\lambda(s,u)}\tilde{N}_{\lambda}(\dif s, \dif u), \quad t\in[0,T].
\de
Thus, we will analyze $M_t$ to get the desired sufficient condition. Firstly, we have
$$
\Delta M_t:=M_t-M_{t-}=\frac{1-\lambda(t,u)}{\lambda(t,u)}=\frac{1}{\lambda(t,u)}-1>-1,\quad t\in[0,T].
$$
Secondly, we assume the following

\begin{enumerate}[{\bf (H2.1)}]
\item
\ce
&&\mE\Big[\exp\Big\{\frac{1}{2}\int_0^T\left\|\gamma(s,x)\right\|_{\tilde{\mH}}^2\dif s
+\int_0^T\int_{\mU_0}\left(\frac{1-\lambda(s,u)}{\lambda(s,u)}\right)^2\lambda(s,u)\nu(\dif u)\dif s\Big\}\Big]\\
&<&\infty.
\de
\end{enumerate}
Then, under {\bf (H2.1)}, $M_t$ is a locally square integrable martingale. Moreover, let $M^c$ and $M^d$ stand for
continuous and purely discontinuous martingale parts of $M$, respectively, then
\ce
&&\mE\Big[\exp\Big\{\frac{1}{2}<M^c,M^c>_T+<M^d,M^d>_T\Big\}\Big]\\
&=&\mE\Big[\exp\Big\{\frac{1}{2}\int_0^T\left\|\gamma(s,x)\right\|_{\tilde{\mH}}^2\dif s
+\int_0^T\int_{\mU_0}\left(\frac{1-\lambda(s,u)}{\lambda(s,u)}\right)^2\lambda(s,u)\nu(\dif u)\dif s\Big\}\Big]\\
&<&\infty.
\de
Thus, it follows from \cite[Theorem 6]{ppks} that $\Lambda_t, t\in[0,T]$ is a exponential martingale and satisfies the
condition (\ref{exma}).

\section{Stochastic evolution equations with jumps on $\mH$}

In the section, we consider stochastic evolution equations with jumps in our infinite dimensional setting. Let us begin with
some notions and notations. For the Hilbert space $\mH$ given in the previous section, $L(\mH)$ is the set of
all bounded linear operators $\cL:$ $\mH\rightarrow\mH$ and $L_{HS}(\mH)$ is the collection of all Hilbert-Schmidt
operator $\cL:\mH\rightarrow\mH$ equipped with the Hilbert-Schmidt norm $\|\cdot\|_{HS}$.

Fix a linear, unbounded, negative definite and self-adjoint operator $(\cA,\cD(\cA))$ on $\mH$, where $\cD(\cA)$ is the domain
of the operator $\cA$. Let $\{e^{t\cA}\}_{t\ge0}$ be the contraction $C_0$-semigroup generated by $\cA$. Moreover, $L_{\cA}(\mH)$
stands for the family of all densely defined closed linear operators $(\cL,\cD(\cL))$ on $\mH$ so that $e^{t\cA}\cL$ can
extend uniquely to a Hilbert-Schmidt operator still denoted by $e^{t\cA}\cL$ for any $t>0$. And then $L_{\cA}(\mH)$, endowed
with the $\sigma$-algebra induced by $\{\cL\to\langle e^{tA}\cL x,y\rangle_{\mH}\mid t>0,x,y\in \mH\}$, becomes a measurable space.

Give $T>0$. Consider the following stochastic evolution equation with jumps on $\mH$
\begin{equation}
\left\{ \begin{aligned}
         d X_t&=\{\cA X_t+b(t,X_t)\}dt+\sigma(t,X_t)dW_t+\int_{\mU_0}f(t,X_{t-},u)\tilde{N_\lambda}(\dif t, \dif u),\ \ \ 0<t\leq T\\
                  X_0&=x_0\in \mH,
                          \end{aligned} \label{3} \right.
                          \end{equation}
where $b:[0,\infty)\times \mH\rightarrow \tilde{\mH}$, $\sigma:[0,\infty)\times \mH\rightarrow L_{\cA}(\mH)$
and $f:[0,\infty)\times\mH\times\mU_0\mapsto\tilde{\mH}$ are all Borel measurable mappings. Set $\|x\|_{\mH}=\infty, x\notin\mH$. Let us give a definition of mild solutions
to Eq.(\ref{3}), which will be used in the sequel.

\bd\label{mildso}
A $\mH$-valued predictable process $X_t, t\in[0,T]$ is called a mild solution of
Eq.(\ref{3}) if for any $t\in[0,T]$
\ce
&&\mE\int_0^t\|e^{(t-s)\cA}b(s,X_s)\|^2_{\mH}\dif s+\mE\int_0^t\|e^{(t-s)\cA}\sigma(s,X_s)\|^2_{HS}\dif s\\
&&+\mE\int_0^t\int_{\mU_0}\|e^{(t-s)\cA}f(s,X_{s-},u)\|^2_{\mH}\lambda(s,u)\nu(\dif u)\dif s<\infty,
\de
and $\mathbb{P}$-a.s.
\ce
X_t&=&e^{t\cA}x_0+\int_0^te^{(t-s)\cA}b(s,X_s)ds+\int_0^te^{(t-s)\cA}\sigma(s,X_s)dW_s\\
&&+\int_0^t\int_{\mU_0}e^{(t-s)\cA}f(s,X_{s-},u)\tilde{N_\lambda}(\dif s, \dif u).
\de
\ed

Next, let us derive existence and uniqueness for the mild solution of Eq.(\ref{3}). To this end, we assume the following
\begin{enumerate}[{\bf (H3.1)}]
\item There exists an integrable function $L_b:(0,T]\to(0,\infty)$ such that
$$
\|e^{s\cA}(b(t,x)-b(t,y))\|^2_{\mH}\leq L_b(s)\|x-y\|^2_{\mH}, \ \ \ s\in(0,T], t\in[0,T], x,y\in\mH,
$$
and
$$
\int^T_0\sup_{r\in[0,T]}\|e^{s\cA}b(r,0)\|^2_{\mH}\dif s<\infty.
$$
\end{enumerate}
\begin{enumerate}[{\bf (H3.2)}]
\item There exists an integrable function $L_{\sigma}:(0,T]\to(0,\infty)$ such that $\forall s\in(0,T], t\in[0,T]$ and $\forall x,y\in \mH$
$$
\|e^{s\cA}\left(\sigma(t,x)-\sigma(t,y)\right)\|^2_{HS}\leq L_{\sigma}(s)\|x-y\|^2_{\mH}$$
and
$$
\int^T_0\sup_{r\in[0,T]}\|e^{s\cA}\sigma(r,0)\|^2_{HS}\dif s<\infty.
$$
\end{enumerate}
\begin{enumerate}[{\bf (H3.3)}]
\item There exists an integrable function $L_f:(0,T]\to(0,\infty)$ such that $\forall s\in(0,T], t\in[0,T]$ and $\forall x,y\in \mH$
$$
\int_{\mU_0}\|e^{s\cA}(f(t,x,u)-f(t,y,u))\|^2_{\mH}\lambda(t,u)\nu(\dif u)\leq L_f(s)\|x-y\|^2_{\mH}$$
and
$$
\int^T_0\int_{\mU_0}\sup_{r\in[0,T]}\left(\|e^{s\cA}f(r,0,u)\|^2_{\mH}\lambda(r,u)\right)\nu(\dif u)\dif s<\infty.
$$
\end{enumerate}

\br\label{asscom}
(i) Under {\bf (H3.1)}-{\bf (H3.3)}, the following hold
\ce
\|e^{s\cA}b(t,x)\|^2_{\mH}&\leq& 2L_b(s)\|x\|^2_{\mH}+2\|e^{s\cA}b(t,0)\|^2_{\mH},\\
\|e^{s\cA}\sigma(t,x)\|^2_{HS}&\leq& 2L_\sigma(s)\|x\|^2_{\mH}+2\|e^{s\cA}\sigma(t,0)\|^2_{HS},\\
\int_{\mU_0}\|e^{s\cA}f(t,x,u)\|^2_{\mH}\lambda(t,u)\nu(\dif u)
&\leq& 2L_f(s)\|x\|^2_{\mH}+2\int_{\mU_0}\|e^{s\cA}f(t,0,u)\|^2_{\mH}\lambda(t,u)\nu(\dif u).
\de
These conditions are nothing but similar to linear growth conditions.

(ii) Comparing {\bf (H3.1)}-{\bf (H3.3)} with those in \cite[Theorem 2.3]{mpr}, one could find that our conditions are more general.
\er

We are  now ready to give the existence and uniqueness result of the mild solution of Eq.(\ref{3})
under {\bf (H3.1)}-{\bf (H3.3)}.

\bt\label{exun}
Suppose that $b, \sigma, f$ satisfy {\bf (H3.1)}-{\bf (H3.3)}. Then there exists a unique mild
solution $X$ of Eq.(\ref{3}) with the following property
$$
\sup_{t\in[0,T]}\mathbb{E}\|X_t\|_{\mH}^2<\infty.
$$
\et
\begin{proof}
Denote by $\cH$ the set of all $\mH$-valued predictable processes $Y=(Y_t)_{t\in[0,T]}$ satisfying
$\sup\limits_{t\in[0,T]}\mE\|Y_t\|_{\mH}^2<\infty$. For $Y\in\cH$, set
\ce
J(Y)(t)&:=&e^{t\cA}x_0+\int_0^te^{(t-s)\cA}b(s,Y_s)ds+\int_0^te^{(t-s)\cA}\sigma(s,Y_s)dW_s\\
&&+\int_0^t\int_{\mU_0}e^{(t-s)\cA}f(s,Y_{s-},u)\tilde{N_\lambda}(\dif s, \dif u),
\de
and then $J(Y)\in\cH$. In fact, by the isometry formula and Remark \ref{asscom}, it holds that for $0\leq s<t\leq T$,
\ce
&&\mE\|\int_0^t\int_{\mU_0}e^{(t-r)\cA}f(r,Y_{r-},u)\tilde{N_\lambda}(\dif r, \dif u)-\int_0^s\int_{\mU_0}e^{(s-r)\cA}f(r,Y_{r-},u)\tilde{N_\lambda}(\dif r, \dif u)\|_{\mH}^2\\
&\leq&2\mE\|\int_0^s\int_{\mU_0}e^{(t-r)\cA}f(r,Y_{r-},u)\tilde{N_\lambda}(\dif r, \dif u)-\int_0^s\int_{\mU_0}e^{(s-r)\cA}f(r,Y_{r-},u)\tilde{N_\lambda}(\dif r, \dif u)\|_{\mH}^2\\
&&+2\mE\|\int_s^t\int_{\mU_0}e^{(t-r)\cA}f(r,Y_{r-},u)\tilde{N_\lambda}(\dif r, \dif u)\|_{\mH}^2\\
&=&2\mE\int_0^s\int_{\mU_0}\|e^{(t-r)\cA}f(r,Y_{r-},u)-e^{(s-r)\cA}f(r,Y_{r-},u)\|_{\mH}^2\lambda(r,u)\nu(\dif u)\dif r\\
&&+2\mE\int_s^t\int_{\mU_0}\|e^{(t-r)\cA}f(r,Y_{r-},u)\|_{\mH}^2\lambda(r,u)\nu(\dif u)\dif r\\
&\leq&2\|e^{(t-s)\cA}-I\|^2\mE\int_0^s\int_{\mU_0}\|e^{(s-r)\cA}f(r,Y_{r-},u)\|_{\mH}^2\lambda(r,u)\nu(\dif u)\dif r\\
&&+2\mE\int_s^t\int_{\mU_0}\|e^{(t-r)\cA}f(r,Y_{r-},u)\|_{\mH}^2\lambda(r,u)\nu(\dif u)\dif r\\
&\leq&4\|e^{(t-s)\cA}-I\|^2\Big(\sup\limits_{t\in[0,T]}\mE\|Y_t\|_{\mH}^2\int_0^sL_f(r)\dif r\Big)\\
&&+4\|e^{(t-s)\cA}-I\|^2\Big(\int_0^s\int_{\mU_0}\sup\limits_{r\in[0,T]}\(\|e^{v\cA}f(r,0,u)\|^2_{\mH}\lambda(r,u)\)\nu(\dif u)\dif v\Big)\\
&&+4\sup\limits_{t\in[0,T]}\mE\|Y_t\|_{\mH}^2\int_0^{t-s}L_f(r)\dif r+4\int_0^{t-s}\int_{\mU_0}\sup\limits_{r\in[0,T]}\(\|e^{v\cA}f(r,0,u)\|^2_{\mH}\lambda(r,u)\)\nu(\dif u)\dif v.
\de
Taking the limits on two sides as $s\rightarrow t$, we obtain that $\int_0^t\int_{\mU_0}e^{(t-r)\cA}f(r,Y_{r-},u)\tilde{N_\lambda}(\dif r, \dif u)$
is mean square continuous in $t$. And
mean square continuity of $e^{t\cA}x_0$, $\int_0^te^{(t-s)\cA}b(s,Y_s)ds$, $\int_0^te^{(t-s)\cA}\sigma(s,Y_s)dW_s$ is easy to verify.
Therefore, $J(Y)(t)$ is mean square continuous in $t$ and then is predictable.

Moreover, it follows from the above definition, the H\"older inequality and the isometry formula that
\ce
&& \mE\|J(Y)(t)\|^2_{\mH}\\
&\leq&4\|e^{t\cA}x_0\|^2_{\mH}+4t\mE\int_0^t\|e^{(t-s)\cA}b(s,Y_s)\|^2_{\mH}\dif s+4\mE\|\int_0^te^{(t-s)\cA}\sigma(s,Y_s)d W_s\|^2_{\mH}\\
&&+4\mE\|\int_0^t\int_{\mU_0}e^{(t-s)\cA}f(s,Y_{s-},u)\tilde{N_\lambda}(\dif s, \dif u)\|^2_{\mH}\\
&\leq&4\|e^{t\cA}x_0\|^2_{\mH}+4t\mE\int_0^t\|e^{(t-s)\cA}b(s,Y_s)\|^2_{\mH}\dif s+4\mE\int_0^t\|e^{(t-s)\cA}\sigma(s,Y_s)\|^2_{HS}\dif s\\
&&+4\mE\int_0^t\int_{\mU_0}\|e^{(t-s)\cA}f(s,Y_{s-},u)\|^2_{\mH}\lambda(s,u)\nu(\dif u)\dif s.
\de
Remark \ref{asscom} then let us to obtain further
\ce
&& \mE\|J(Y)(t)\|^2_{\mH}\\
&\leq&4\|e^{t\cA}x_0\|^2_{\mH}+8t\int_0^tL_b(t-s)\mE\|Y_s\|^2_{\mH}\dif s+8\int_0^tL_\sigma(t-s)\mE\|Y_s\|^2_{\mH}\dif s\\
&&+8\int_0^tL_f(t-s)\mE\|Y_s\|^2_{\mH}\dif s+8\int_0^t\bigg[t\|e^{(t-s)\cA}b(s,0)\|^2_{\mH}+\|e^{(t-s)\cA}\sigma(s,0)\|^2_{HS}\\
&&+\int_{\mU_0}\|e^{(t-s)\cA}f(s,0,u)\|^2_{\mH}\lambda(s,u)\nu(\dif u)\bigg]\dif s\\
&\leq&4\|e^{t\cA}x_0\|^2_{\mH}+8\sup\limits_{t\in[0,T]}\mE\|Y_t\|_{\mH}^2\left(t\int_0^tL_b(s)\dif s+\int_0^tL_\sigma(s)\dif s+\int_0^tL_f(s)\dif s\right)\\
&&+8\int_0^t\bigg[t\sup_{r\in[0,T]}\|e^{s\cA}b(r,0)\|^2_{\mH}+\sup_{r\in[0,T]}\|e^{s\cA}\sigma(r,0)\|^2_{HS}\\
&&+\int_{\mU_0}\sup_{r\in[0,T]}\(\|e^{s\cA}f(r,0,u)\|^2_{\mH}\lambda(r,u)\)\nu(\dif u)\bigg]\dif s.
\de
By {\bf (H3.1)} {\bf (H3.2)} {\bf (H3.3)} again, it holds that $\sup\limits_{t\in[0,T]}\mE\|J(Y)(t)\|_{\mH}^2<\infty$.

Next, let us calculate $\sup\limits_{s\in[0,t]}\mE\|J(Y^1)(s)-J(Y^2)(s)\|_{\mH}^2$ for $Y^1, Y^2\in\cH$. By similar derivation
to the above, one could have
\ce
\sup\limits_{s\in[0,t]}\mE\|J(Y^1)(s)-J(Y^2)(s)\|_{\mH}^2&\leq&3\left(t\int_0^tL_b(r)\dif r+\int_0^tL_\sigma(r)\dif r+\int_0^tL_f(r)\dif r\right)\\
&&\cdot\sup\limits_{s\in[0,t]}\mE\|Y^1_s-Y^2_s\|_{\mH}^2.
\de
Since $\lim\limits_{t\rightarrow0}3\left(t\int_0^tL_b(r)\dif r+\int_0^tL_\sigma(r)\dif r+\int_0^tL_f(r)\dif r\right)=0$,
there exists a $0<t_0\leq T$ such that $3\left(t_0\int_0^{t_0}L_b(r)\dif r+\int_0^{t_0}L_\sigma(r)\dif r+\int_0^{t_0}L_f(r)\dif r\right)<1$.
Thus on $[0,t_0]$ the mapping $J$ has a unique fixed point $Y$ which is a unique mild solution of Eq.(\ref{3}).
If $t_0=T$, the proof is finished. If $t_0<T$, we repeat the above procedure to get a unique mild solution of Eq.(\ref{3})
on $[t_0, t_1]$ for some $t_1\in(t_0, T]$. The approach is further utilsed till $t_n=T$ so that a unique mild solution of Eq.(\ref{3})
on the whole interval $[0, T]$ is obtained. We are done.
\end{proof}

Next, we will construct a finite dimensional approximation to Eq.\eqref{3} to set up a relation between Eq.\eqref{3}
and a finite dimensional SDE with jumps. To be more precise, we will set up the Galerkin approximation to Eq.\eqref{3}, for
which we shall need the following assumption:
\begin{enumerate}[{\bf (H3.4)}]
\item The operator $-\cA$ has the following eigenvalues
$$0<\lambda_1\leq\lambda_2\leq\ldots\leq\lambda_j\leq\ldots$$
counting multiplicities such that
$$\sum^\infty_{j=1}\frac1{\lambda_j}<\infty\, .$$
We would like to emphasize here that the complete orthonormal basis $\{e_j\}_{j\in\mathbb{N}}$ is taken as the eigen-basis of $-\cA$
throughout the rest of the paper.
\end{enumerate}

\begin{remark}
Note that by {\bf (H3.4)}, there are invertible operators on $\mH$ in $L_{\cA}(\mH)$, such as the identity operator $I$.
\end{remark}

Recall that from now on we have the fixed complete orthonormal basis $\{e_j\}_{j\in\mathbb{N}}$ for $\mH$ as specified in
{\bf(H3.4)}. Set
\ce
&&\pi_n:\mH\rightarrow {\mH}_n:=\mbox{span}\{e_1,\cdots,e_n\},  \quad n\in\mN,\\
&&\pi_nx:=\sum_{i=1}^n\langle x,e_i\rangle_{\mH}e_i, \quad x\in\mH,
\de
and then $\pi_n$ is the orthogonal project operator from $\mH$ to ${\mH}_n$. Moreover, $\pi_ne^{t\cA}=e^{t\cA}\pi_n$ for
$t\ge0$. Again set $\cA_n:=\cA\mid_{{\mH}_n}, b_n:=\pi_n b$, $\sigma_n:=\pi_n\sigma$ and $f_n:=\pi_n f$. And then consider the following SDE with jumps in ${\mH}_n$
\begin{equation}\label{8}
\begin{cases}
d X^n_t=\{\cA_nX^n_t+b_n(t,X^n_t)\}dt+\sigma_n(t,X^n_t)dW_t+\int_{\mU_0}f_n(t,X^n_{t-},u)\tilde{N_\lambda}(\dif t, \dif u),\\
X^n(0)=\pi_nx_0.
\end{cases}
\end{equation}
Under {\bf (H3.1)}-{\bf (H3.4)}, one can justify that the coefficients $b_n$, $\sigma_n$ and $f_n$ are Lipschitz continuous
and linearly growing. For example, for $\sigma_n$, noting that $W_t=\sum_{j=1}^\infty\beta_j(t,\omega)e_j$,
we deduce that
\ce
&& \sum_{j=1}^\infty\|\sigma_n(t,x)e_j\|_{\mH_n}^2\\
&=&\sum_{j=1}^\infty\sum_{i=1}^n|\langle\sigma(t,x)e_j,e_i\rangle_{\mH}|^2
=\sum_{j=1}^\infty\sum_{i=1}^ne^{2\lambda_iT}|\langle\sigma(t,x)e_j,e^{-\lambda_iT}e_i\rangle_{\mH}|^2\\
&=&\sum_{j=1}^\infty\sum_{i=1}^ne^{2\lambda_iT}|\langle\sigma(t,x)e_j,e^{T\cA}e_i\rangle_{\mH}|^2
\leq e^{2\lambda_nT}\sum_{j=1}^\infty\sum_{i=1}^n|\langle e^{T\cA}\sigma(t,x)e_j,e_i\rangle_{\mH}|^2\\
&\leq&e^{2\lambda_nT}\|e^{T\cA}\sigma(t,x)\|^2_{HS}\leq e^{2\lambda_nT}\left(2L_\sigma(T)\|x\|^2_{\mH}+2\|e^{T\cA}\sigma(t,0)\|^2_{HS}\right),
\de
i.e. $\sigma_n$ is linearly growing. Thus, by \cite[Theorem 9.1]{iw}, there exists a unique strong solution
$X^n_t\in {\mH}_n,t\in[0,T]$ to Eq.\eqref{8}. Moreover, we have the following result.

\bl\label{galeapp}
Under {\bf (H3.1)}-{\bf (H3.4)},
\begin{equation}\label{eq2}
\lim_{n\rightarrow\infty}\mathbb{E}\|X_t^n-X_t\|_{\mH}^2=0, \quad t\in[0,T].
\end{equation}
\el
\begin{proof}
Note that $X^n$, the unique strong solution to Eq.\eqref{8}, also satisfies the following equation
\be
X^n_t&=&e^{t\cA_n}\pi_nx_0+\int_0^te^{(t-s)\cA_n}b_n(s,X^n_s)ds+\int_0^te^{(t-s)\cA_n}\sigma_n(s,X^n_s)dW_s\no\\
&&+\int_0^t\int_{\mU_0}e^{(t-s)\cA_n}f_n(s,X^n_{s-},u)\tilde{N_\lambda}(\dif s, \dif u), \quad t\in[0,T].
\label{appmil}
\ee
Based on this and Definition \ref{mildso}, we compute $\mathbb{E}\|X_t^n-X_t\|_{\mH}^2$. By the H\"older inequality
and the isometry formula, it holds that
\ce
&& \mathbb{E}\|X_t^n-X_t\|_{\mH}^2\\
&\leq&4\|e^{t\cA_n}\pi_n x_0-e^{t\cA}x_0\|^2_{\mH}+4t\mE\int_0^t\|e^{(t-s)\cA_n}b_n(s,X^n_s)-e^{(t-s)\cA}b(s,X_s)\|^2_{\mH}\dif s\\
&&+4\mE\int_0^t\|e^{(t-s)\cA_n}\sigma_n(s,X^n_s)-e^{(t-s)\cA}\sigma(s,X_s)\|^2_{HS}\dif s\\
&&+4\mE\int_0^t\int_{\mU_0}\|e^{(t-s)\cA_n}f_n(s,X^n_{s-},u)-e^{(t-s)\cA}f(s,X_{s-},u)\|^2_{\mH}\lambda(s,u)\nu(\dif u)\dif s\\
&=:&I_1+I_2+I_3+I_4.
\de

For $I_1$, since $e^{t\cA_n}\pi_n x_0=\pi_ne^{t\cA}x_0$,
\be\label{i1}
I_1\leq 4\|\pi_n-I\|^2\|e^{t\cA}x_0\|^2_{\mH}.
\ee
And then we deal with $I_2$. It follows from {\bf (H3.1)} that
\be\label{i2}
I_2&\leq&4t\mE\int_0^t\bigg[2\|e^{(t-s)\cA_n}b_n(s,X^n_s)-e^{(t-s)\cA_n}b_n(s,X_s)\|^2_{\mH}\no\\
&&+2\|e^{(t-s)\cA_n}b_n(s,X_s)-e^{(t-s)\cA}b(s,X_s)\|^2_{\mH}\bigg]\dif s\no\\
&\leq&4t\mE\int_0^t\bigg[2\|e^{(t-s)\cA}b(s,X^n_s)-e^{(t-s)\cA}b(s,X_s)\|^2_{\mH}\no\\
&&+2\|\pi_n-I\|^2\|e^{(t-s)\cA}b(s,X_s)\|^2_{\mH}\bigg]\dif s\no\\
&\leq&8t\int_0^tL_b(t-s)\mE\|X_s^n-X_s\|_{\mH}^2\dif s\no\\
&&+8t\|\pi_n-I\|^2\int_0^t\mE\|e^{(t-s)\cA}b(s,X_s)\|^2_{\mH}\dif s.
\ee
By the similar deduction to the above, we obtain that
\be
I_3+I_4&\leq&8\int_0^t\left(L_\sigma(t-s)+L_f(t-s)\right)\mE\|X_s^n-X_s\|_{\mH}^2\dif s\no\\
&&+8\|\pi_n-I\|^2\int_0^t\mE\|e^{(t-s)\cA}\sigma(s,X_s)\|^2_{HS}\dif s\no\\
&&+8\|\pi_n-I\|^2\int_0^t\int_{\mU_0}\mE\|e^{(t-s)\cA}f(s,X_{s-},u)\|^2_{HS}\lambda(s,u)\nu(\dif u)\dif s.\label{i34}
\ee
Combining (\ref{i1}) (\ref{i2}) with (\ref{i34}), we have further that
\be
\mathbb{E}\|X_t^n-X_t\|_{\mH}^2&\leq&4\|\pi_n-I\|^2\|e^{t\cA}x_0\|^2_{\mH}+8\|\pi_n-I\|^2U_t\no\\
&&+8\int_0^t\left(TL_b(t-s)+L_\sigma(t-s)+L_f(t-s)\right)\mE\|X_s^n-X_s\|_{\mH}^2\dif s,
\label{i}
\ee
where
\ce
U_t&:=&t\int_0^t\mE\|e^{(t-s)\cA}b(s,X_s)\|^2_{\mH}\dif s+\int_0^t\mE\|e^{(t-s)\cA}\sigma(s,X_s)\|^2_{HS}\dif s\\
&&\qquad +\int_0^t\int_{\mU_0}\mE\|e^{(t-s)\cA}f(s,X_{s-},u)\|^2_{HS}\lambda(s,u)\nu(\dif u)\dif s.
\de
By Definition \ref{mildso} and Theorem \ref{exun}, it holds that $U_t<\infty$ and $\sup\limits_{t\in[0,T]}\mathbb{E}\|X_t\|_{\mH}^2<\infty$.

Next, we compute $\sup\limits_{n\geq 1}\sup\limits_{t\in[0,T]}\mathbb{E}\|X^n_t\|_{\mH}^2$. For Eq.(\ref{appmil}),
by the similar calculation to that in the proof of Theorem \ref{exun}, one could have that
\ce
\mE\|X^n_t\|^2_{\mH}&\leq&4\|e^{t\cA}x_0\|^2_{\mH}+8t\int_0^tL_b(t-s)\mE\|X^n_s\|^2_{\mH}\dif s+8\int_0^tL_\sigma(t-s)\mE\|X^n_s\|^2_{\mH}\dif s\\
&&+8\int_0^tL_f(t-s)\mE\|X^n_s\|^2_{\mH}\dif s+8\int_0^t\bigg[t\|e^{(t-s)\cA}b(s,0)\|^2_{\mH}+\|e^{(t-s)\cA}\sigma(s,0)\|^2_{HS}\\
&&+\int_{\mU_0}\|e^{(t-s)\cA}f(s,0,u)\|^2_{\mH}\lambda(s,u)\nu(\dif u)\bigg]\dif s\\
&\leq&4\|e^{t\cA}x_0\|^2_{\mH}+8\sup\limits_{s\in[0,t]}\mE\|X^n_s\|_{\mH}^2\left(t\int_0^tL_b(s)\dif s+\int_0^tL_\sigma(s)\dif s+\int_0^tL_f(s)\dif s\right)\\
&&+8\int_0^t\bigg[t\sup_{r\in[0,T]}\|e^{s\cA}b(r,0)\|^2_{\mH}+\sup_{r\in[0,T]}\|e^{s\cA}\sigma(r,0)\|^2_{HS}\\
&&+\int_{\mU_0}\sup_{r\in[0,T]}\(\|e^{s\cA}f(r,0,u)\|^2_{\mH}\lambda(r,u)\)\nu(\dif u)\bigg]\dif s.
\de
Furthermore, by taking $t_0$ with $8\left(t_0\int_0^{t_0}L_b(s)\dif s+\int_0^{t_0}L_\sigma(s)\dif s+\int_0^{t_0}L_f(s)\dif s\right)<1/2$,
it holds that
\ce
\sup\limits_{s\in[0,t_0]}\mE\|X^n_s\|^2_{\mH}&\leq&8\|e^{t_0\cA}x_0\|^2_{\mH}+16\int_0^{t_0}\bigg[t_0\sup_{r\in[0,T]}\|e^{s\cA}b(r,0)\|^2_{\mH}+\sup_{r\in[0,T]}\|e^{s\cA}\sigma(r,0)\|^2_{HS}\\
&&+\int_{\mU_0}\sup_{r\in[0,T]}\(\|e^{s\cA}f(r,0,u)\|^2_{\mH}\lambda(r,u)\)\nu(\dif u)\bigg]\dif s.
\de
On $[t_0,2t_0], [2t_0,3t_0], \dots, [mt_0,T]$ for $m\in\mN$, by the same way to the above we deduce and conclude that
\ce
\sup\limits_{s\in[0,T]}\mE\|X^n_s\|^2_{\mH}&\leq&8^{m+1}\|e^{t_0\cA}x_0\|^2_{\mH}+16\sum\limits_{i=1}^m8^i\int_{(i-1)t_0}^{it_0}\bigg[t_0\sup_{r\in[0,T]}\|e^{s\cA}b(r,0)\|^2_{\mH}\\
&&+\sup_{r\in[0,T]}\|e^{s\cA}\sigma(r,0)\|^2_{HS}+\int_{\mU_0}\sup_{r\in[0,T]}\|e^{s\cA}f(r,0,u)\|^2_{\mH}\lambda(r,u)\nu(\dif u)\bigg]\dif s\\
&&+16\int_{mt_0}^{T}\bigg[t_0\sup_{r\in[0,T]}\|e^{s\cA}b(r,0)\|^2_{\mH}+\sup_{r\in[0,T]}\|e^{s\cA}\sigma(r,0)\|^2_{HS}\\
&&+\int_{\mU_0}\sup_{r\in[0,T]}\(\|e^{s\cA}f(r,0,u)\|^2_{\mH}\lambda(r,u)\)\nu(\dif u)\bigg]\dif s.
\de
This shows $\sup\limits_{n\geq 1}\sup\limits_{t\in[0,T]}\mathbb{E}\|X^n_t\|_{\mH}^2<\infty$.

Thus, taking the super limit on two sides of the inequality (\ref{i}) as $n\rightarrow\infty$, by the Fatou lemma we obtain that
\ce
\limsup_{n\rightarrow\infty}\mathbb{E}\|X_t^n-X_t\|_{\mH}^2
&\leq& \int_0^t\left(TL_b(t-s)+L_\sigma(t-s)\right. \\
&&\quad \left.+L_f(t-s)\right)\limsup_{n\rightarrow\infty}\mE\|X_s^n-X_s\|_{\mH}^2\dif s.
\de
Based on the proof in \cite[Theorem 3.1.2]{FYWang}, one could get
$$
\limsup_{n\rightarrow\infty}\mathbb{E}\|X_t^n-X_t\|_{\mH}^2=0.
$$
Thus, the proof is completed.
\end{proof}

In the following, we will apply the Galerkin finite-dimensional approximation in the above lemma to deduce
an It\^o formula for real-valued functions of the solution $X_t, t\in[0,T]$ to Eq.\eqref{3}. Now, there exist
some It\^o formulas for real-valued functions of these solutions processes for these infinite-dimensional
semi-linear SDEs with jumps containing Eq.\eqref{3}, such as \cite[Theorem 2.4]{bty} and \cite[Theorem 27.1]{mm}. Unfortunately,
they don't work here because of two requirements in them that the diffusion coefficient $\sigma$ in Eq.\eqref{3}
is a Hilbert-Schmidt operator and that the solution $X_t$ to Eq.\eqref{3} is a strong solution. In present,
we will prove an It\^o formula for Eq.\eqref{3} with $\sigma\in L_{\cA}(\mH)$ and a unique mild solution $X_t$.
Therefore, the result is independently interesting.

\begin{proposition}\label{itoform}
Assume {\bf (H3.1)}-{\bf (H3.4)}, and let $v: [0,T]\times \mH\to\mathbb{R}$ be in $C_b^{1,2}([0,T]\times\mH)$ such
that $[\nabla v(t,x)]\in\cD(\cA)$ for any $(t,x)\in[0,T]\times \mH$ and $\|\cA\nabla v(t,\cdot)\|_{\mH}$ is
bounded locally and uniformly in $t\in[0,T]$. Then we have
\begin{eqnarray}\label{17}
v(t,X_t) &=&v(0,x_0)+\int_0^t\bigg[\frac{\partial}{\partial s}v(s,X_s)+\langle\nabla v(s,X_s), b(s,X_s)\rangle_{\tilde{\mH}}
+\langle \cA\nabla v(s,X_s),X_s\rangle_{\mH}\bigg]ds\no\\
&&+\int_0^t\langle\sigma^*(s,X_s)\nabla v(s,X_s),dW_s\rangle_{\tilde{\mH}}+ \frac{1}{2}\int^t_0Tr[(\sigma\sigma^*)(s,X_s)\nabla^2v(s,X_s)]ds\no\\
&& +\int_0^t\int_{\mU_0}\left[v(s,X_{s-}+f(s,X_{s-},u))-v(s,X_{s-})\right]\tilde{N_\lambda}(\dif s, \dif u)\no\\
&&+\int_0^t\int_{\mU_0}\Big[v(s,X_{s-}+f(s,X_{s-},u))-v(s,X_{s-})\no\\
&&\quad -\<f(s,X_{s-},u),\nabla v(s,X_{s-})\>_{\tilde{\mH}}\Big]\lambda(s,u)\nu(\dif u)\dif s,
\end{eqnarray}
where $\sigma^*(t,x)$ stands for the transposed matrix of $\sigma(t,x)$, $\nabla$ and $\nabla^2$ stand for the
first and second {\it Fr\'{e}chet} operators with respect to the second variable, respectively.
\end{proposition}
\begin{proof}
First of all, take the approximation sequence $\{X^n_t, t\in[0,T]\}_{n\in \mathbb{N}}$ for the solution $\{X_t, t\in[0,T]\}$
of Eq.\eqref{3}. Note that $\{X^n_t, t\in[0,T]\}$ is a $n$-dimensional process. Applying the It\^o formula in \cite{iw}
to $v(t,X^n_t)$ for any $t\in[0,T]$, one could obtain that
\begin{eqnarray}\label{itoF}
v(t,X_t^n)&=&v(0,\pi_n x_0)+ \int_0^t\langle\nabla_nv(s,X^n_s),\sigma_n(s,X_s^n)dW_s\rangle_{\mH}\nonumber\\
&& +\int_0^t\int_{\mU_0}\left[v(s,X^n_{s-}+f_n(s,X^n_{s-},u))-v(s,X^n_{s-})\right]\tilde{N_\lambda}(\dif s, \dif u)\no\\
&&\quad +\int_0^t\big[\frac{\partial}{\partial s}v(s,X^n_s)+\langle\nabla_nv(s,X^n_s),\cA_nX^n_s+b_n(s,X_s^n)\rangle_{\mH}\big]ds\nonumber\\
&&\quad +\frac{1}{2} \int^t_0Tr[\nabla^2_nv(s,X^n_s)(\sigma_n(s,X_s^n)(Id)^{\frac12})(\sigma_n(s,X_s^n)(Id)^{\frac12})^*]ds\nonumber\\
&&+\int_0^t\int_{\mU_0}\Big[v(s,X^n_{s-}+f_n(s,X^n_{s-},u))-v(s,X^n_{s-})\no\\
&&\quad -\<f_n(s,X^n_{s-},u),\nabla_n v(s,X^n_{s-})\>_{\mH}\Big]\lambda(s,u)\nu(\dif u)\dif s\no\\
&=&v(0,\pi_n x_0)+ \int_0^t\langle\sigma^*_n(s,X_s^n)\nabla_nv(s,X^n_s),dW_s\rangle_{\mH}\nonumber\\
&& +\int_0^t\int_{\mU_0}\left[v(s,X^n_{s-}+f_n(s,X^n_{s-},u))-v(s,X^n_{s-})\right]\tilde{N_\lambda}(\dif s, \dif u)\no\\
&&\quad +\int_0^t\big[\frac{\partial}{\partial s}v(s,X^n_s)+\langle\nabla_nv(s,X^n_s),\cA_nX^n_s+b_n(s,X_s^n)\rangle_{\mH}\big]ds\nonumber\\
&&\quad + \frac{1}{2}\int^t_0Tr[(\sigma_n\sigma_n^*)(s,X_s^n)\nabla^2_nv(s,X^n_s)]ds\no\\
&&+\int_0^t\int_{\mU_0}\Big[v(s,X^n_{s-}+f_n(s,X^n_{s-},u))-v(s,X^n_{s-})\no\\
&&\quad -\<f_n(s,X^n_{s-},u),\nabla_n v(s,X^n_{s-})\>_{\mH}\Big]\lambda(s,u)\nu(\dif u)\dif s,
\end{eqnarray}
where $\nabla_n\cdot :=\sum^n_{j=1}\langle\nabla\cdot,e_j\rangle_{\mH}e_j$.

Firstly, by continuity of $v(t,x), \frac{\partial}{\partial s}v(s,x)$ with respect to $x$ and Lemma \ref{galeapp}, it is clear that
\ce
\lim_{n\to\infty}v(t,\pi_nx_0)&=&v(t,x_0),\\
\lim_{n\to\infty}v(t,X^n_t)&=&v(t,X_t),\\
\lim_{n\to\infty}\frac{\partial}{\partial s}v(s,X^n_s)&=&\frac{\partial}{\partial s}v(s,X_s)\, \quad a.s..
\de
Those assumptions on $v$ and self-adjoint property of the operator $\cA$ admit us to obtain that
\ce
&&\lim_{n\to\infty}\int^t_0\langle[\sigma^*_n\nabla_nv](s,X^n_s),dW_s\rangle_{\mH}
=\int^t_0\langle[\sigma^*\nabla v](s,X_s),dW_s\rangle_{\tilde{\mH}}\, ,\\
&&\lim_{n\to\infty}\int_0^t\int_{\mU_0}\left[v(s,X^n_{s-}+f_n(s,X^n_{s-},u))-v(s,X^n_{s-})\right]\tilde{N_\lambda}(\dif s, \dif u)\\
&=&\int_0^t\int_{\mU_0}\left[v(s,X_{s-}+f(s,X_{s-},u))-v(s,X_{s-})\right]\tilde{N_\lambda}(\dif s, \dif u),
\de
in the mean square sense and
\ce
&&\lim_{n\to\infty}\int^t_0\langle\nabla_nv(s,X^n_s),\cA_nX^n_s\rangle_{\mH}ds
=\int^t_0\langle \cA\nabla v(s,X_s),X_s\rangle_{\mH}ds\, ,\\
&&\lim_{n\to\infty}\int^t_0\langle\nabla_nv(s,X^n_s),b_n(s,X_s^n)\rangle_{\mH}ds
= \int^t_0\langle\nabla v(s,X_s),b(s,X_s)\rangle_{\tilde{\mH}}ds\, ,\\
&&\lim_{n\to\infty}\int^t_0Tr[(\sigma_n\sigma_n^*)(s,X_s^n)\nabla^2_nv(s,X^n_s)]ds
=\int^t_0Tr[(\sigma\sigma^*)(s,X_s)\nabla^2v(s,X_s)]ds,\\
&&\lim_{n\to\infty}\int_0^t\int_{\mU_0}\Big[v(s,X^n_{s-}+f_n(s,X^n_{s-},u))-v(s,X^n_{s-})\\
&&\quad -\<f_n(s,X^n_{s-},u),\nabla_n v(s,X^n_{s-})\>_{\mH}\Big]\lambda(s,u)\nu(\dif u)\dif s\\
&=&\int_0^t\int_{\mU_0}\Big[v(s,X_{s-}+f(s,X_{s-},u))-v(s,X_{s-})\\
&&\quad -\<f(s,X_{s-},u),\nabla v(s,X_{s-})\>_{\tilde{\mH}}\Big]\lambda(s,u)\nu(\dif u)\dif s, \quad a.s..
\de
Finally, taking the limits on two sides of \eqref{itoF} as $n\to\infty$, by the above equalities we have
\eqref{17}. The proof is completed.
\end{proof}

\section{The characterization theorem on $\mH$}\label{chthh}

In the section, we shall state and prove two characterization theorems, which are the main results in the paper.
First of all, consider Eq.(\ref{3}), i.e.
\ce\left\{\begin{array}{l}
\dif X_t=(\cA X_t+b(t,X_t))\dif t+\sigma(t,X_t)dW_t+\int_{\mU_0}f(t,X_{t-},u)\tilde{N_\lambda}(\dif t, \dif u),\ \ \ t\in[0,T],\\
                  X_0=x_0\in \mH.
\end{array}
\right.
\de
Moreover, we assume:
\begin{enumerate}[{\bf (H3.3')}]
\item There exists an integrable function $L'_f:[0,T]\to(0,\infty)$ such that
$$
\int_{\mU_0}\|e^{s\cA}(f(t,x,u)-f(t,y,u))\|^2_{\mH}\lambda(t,u)\nu(\dif u)\leq L'_f(s)\|x-y\|^2_{\mH}, \ \ \ s,t\in[0,T], x,y\in \mH,
$$
and for $q=2$ and $4$
$$
\int_{\mU_0}\|e^{s\cA}f(t,x,u)\|^q_{\mH}\lambda(t,u)\nu(\dif u)\leq L'_f(s)(1+\|x\|_{\mH})^q.
$$
\end{enumerate}
\begin{enumerate}[{\bf (H4.1)}]
\item (Non-degeneracy) $\sigma(t,x)$ is invertible for any $(t,x)\in[0,T]\times\mH$ and the inverse operator of $\sigma(t,x)$ is uniformly bounded on $(t,x)\in[0,T]\times\mH$.
\end{enumerate}

It is obvious that the assumption {\bf (H3.3')} is stronger than {\bf (H3.3)}. In the following, we define the support of
a $\mH$-valued random variable (\cite{pa}) and give out a support theorem under these assumptions.

\bd\label{supset}
The support of a $\mH$-valued random variable $Y$ is defined to be
$$
supp(Y):=\{x\in\mH | (\mP\circ Y^{-1})(B(x,r))>0, ~\mbox{for all}~ r>0\}
$$
where $B(x,r):=\{y\in\mH | \|y-x\|_{\mH}<r\}$.
\ed

\bl\label{supthe}
Under {\bf (H3.1)}-{\bf (H3.2)} {\bf (H3.3')} {\bf (H3.4)} and {\bf (H4.1)}, $supp(X_t)=\mH$ for $t\in[0,T]$.
\el
\begin{proof}
Since it is easy to see $supp(X_t)\subset\mH$, we only prove $supp(X_t)\supset\mH$. Moreover, from Definition \ref{supset},
we only need to show that for any $x\in\mH$ and $r>0$,
$$
\mP\{\|X_t-x\|_{\mH}<r\}>0,
$$
or equivalently,
$$
\mP\{\|X_t-x\|_{\mH}\geq r\}<1.
$$

On one hand, by Lemma \ref{galeapp} and the Chebyshev inequality, it holds that for any small $0<\e<r$ and $0<\eta<1$,
there exists a $N\in\mN$ such that for $n>N$,
$$
\mP\{\|X_t-X^n_t\|_{\mH}\geq \e/2\}<\eta/2, \quad \mP\{\|\pi_nx-x\|_{\mH}\geq \e/2\}<\eta/2.
$$
On the other hand, for $X^n_t$, \cite[Proposition 2.4]{qiao} admits us to obtain that
$$
\mP\{\|X^n_t-\pi_nx\|_{\mH}\geq r-\e\}<1-\eta.
$$
Thus, combining these inequalities, we furthermore have that
\ce
\mP\{\|X_t-x\|_{\mH}\geq r\}&\leq&\mP\{\|X_t-X^n_t\|_{\mH}\geq \e/2\}+\mP\{\|X^n_t-\pi_nx\|_{\mH}\geq r-\e\}\\
&&+\mP\{\|\pi_nx-x\|_{\mH}\geq \e/2\}\\
&<&1.
\de
So, the proof is completed.
\end{proof}

In order to give main results, we also need the following assumption.
\begin{enumerate}[{\bf (H4.2)}]
 \item
\ce
&&\mE\Big[\exp\Big\{\frac{1}{2}\int_0^T\left\|\sigma^{-1}(s,X_s)b(s,X_s)\right\|_{\mH}^2\dif s
+\int_0^T\int_{\mU_0}\left(\frac{1-\lambda(s,u)}{\lambda(s,u)}\right)^2\lambda(s,u)\nu(\dif u)\dif s\Big\}\Big]\\
&<&\infty.
\de
\end{enumerate}

Taking
\ce
\Lambda_t&=&\exp\bigg\{-\int_0^t\<\sigma^{-1}(s,X_s)b(s,X_s),\dif W_s\>_{\tilde{\mH}}-\frac{1}{2}\int_0^t
\left\|\sigma^{-1}(s,X_s)b(s,X_s)\right\|_{\mH}^2\dif s\\
&&\quad\qquad -\int_0^t\int_{\mU_0}\log\lambda(s,u)N_{\lambda}(\dif s, \dif u)
-\int_0^t\int_{\mU_0}(1-\lambda(s,u))\nu(\dif u)\dif s\bigg\},
\de
by Section \ref{sgir} we know that $\Lambda_t$ is a exponential martingale under {\bf (H4.2)}
and satisfies the condition (\ref{exma}). Thus, by Theorem \ref{tgir}, one can obtain that under the measure $\hat{\mP}$
Eq.(\ref{3}) is transformed as
\ce
\dif X_t=\cA X_t\dif t+\sigma(t,X_t)\dif \tilde{W}_t+\int_{\mU_0}f(t,X_{t-},u)\tilde{N}(\dif t, \dif u),
\de
where
\ce
\tilde{W}_t:=W_t+\int_0^t\sigma^{-1}(s,X_s)b(s,X_s)\dif s.
\de

Next, we observe $\Lambda_t$. Set
\ce
Y_t&:=&-\log\Lambda_t\\
&=&\int_0^t\<\sigma^{-1}(s,X_s)b(s,X_s),\dif W_s\>_{\tilde{\mH}}+\frac{1}{2}\int_0^t
\left\|\sigma^{-1}(s,X_s)b(s,X_s)\right\|_{\mH}^2\dif s\\
&&\quad +\int_0^t\int_{\mU_0}\log\lambda(s,u)N_{\lambda}(\dif s, \dif u)
+\int_0^t\int_{\mU_0}(1-\lambda(s,u))\nu(\dif u)\dif s.
\de
Clearly, $Y_t$ is a one-dimensional stochastic process with the stochastic differential form
\ce
\dif Y_t&=&\<\sigma^{-1}(t,X_t)b(t,X_t),\dif W_t\>_{\tilde{\mH}}+\frac{1}{2}\left\|\sigma^{-1}(t,X_t)b(t,X_t)\right\|_{\mH}^2\dif t\\
&&\quad +\int_{\mU_0}\log\lambda(t,u)N_{\lambda}(\dif t, \dif u)
+\int_{\mU_0}(1-\lambda(t,u))\nu(\dif u)\dif t.
\de

Now, we state and prove the first result of the section.

\bt\label{chth1}
Assume {\bf (H3.1)}-{\bf (H3.2)} {\bf (H3.3')} {\bf (H3.4)} and {\bf (H4.1)}-{\bf (H4.2)}. Let $v:[0,T]\times\mH\rightarrow\mR$ be a
scalar function which is $C^1$ with respect to the first variable and $C^2$ with respect to the second variable
such that $[\nabla v(t,x)]\in\cD(\cA)$ for any $(t,x)\in[0,T]\times \mH$ and $\|\cA\nabla v(t,\cdot)\|_{\mH}$ is
bounded locally and uniformly in $t\in[0,T]$, and $\|\cA\nabla v(t,\cdot)\|_{\mH}:\mH\to[0,\infty)$ is continuous (in the variable
$x\in{\mH}$) for each $t\in[0,T]$. Then the Girsanov density $\Lambda_t$ for Eq.\eqref{3}
has the following path-independent property:
\be
\Lambda_t=\exp\{v(0,x_0)-v(t,X_t)\}, \quad t\in[0,T],
\label{pathin}
\ee
if and only if
\be
b(t,x)&=&(\sigma\sigma^*\nabla v)(t, x), \qquad\qquad\qquad\qquad\quad  (t,x)\in[0,T]\times\mH, \label{chco1}\\
\lambda(t,u)&=&\exp\{v(t,x+f(t,x,u))-v(t,x)\}, \quad (t,x,u)\in[0,T]\times\mH\times\mU_0, \label{chco2}
\ee
and $v$ satisfies the following time-reversed integro-differential equation(IDE),
\be
&&\frac{\partial}{\partial t}v(t,x)\no \\
&=&-\frac{1}{2}[Tr(\sigma\sigma^*)\nabla^2 v](t,x)-\frac{1}{2}\|\sigma(t,x)^*\nabla v(t,x)\|_{\mH}^2
-\langle x,\cA\nabla v(t,x)\rangle_{\mH}\no\\
&&-\int_{\mU_0}\Big[e^{v(t,x+f(t,x,u))-v(t,x)}-1-\<f(t,x,u),\nabla v(t,x)\>_{\tilde{\mH}}e^{v(t,x+f(t,x,u))-v(t,x)}\Big]\nu(\dif u).\no\\
\label{chco3}
\ee
\et
\begin{proof}
Firstly, let us show sufficiency. Assume that there exists a $C^{1,2}$-function $v(t,x)$ satisfying
(\ref{chco1})(\ref{chco2})(\ref{chco3}). For the composition process $v(t,X_t)$, the It\^o formula in Proposition \ref{itoform}
admits us to get
\be
\dif v(t, X_t)&=&\frac{\partial}{\partial t}v(t,X_t)\dif t+\<\cA X_t,\nabla v(t, X_t)\>_{\mH}\dif t\no\\
&&+\<b(t, X_t),\nabla v(t, X_t)\>_{\tilde{\mH}}\dif t
+\frac{1}{2}[Tr(\sigma\sigma^*)\nabla^2 v](t, X_t)\dif t\no\\
&&+\int_{\mU_0}\Big[v(t,X_{t-}+f(t,X_{t-},u))-v(t,X_{t-})\no\\
&&\qquad -\<f(t,X_{t-},u),\nabla v(t,X_{t-})\>_{\tilde{\mH}}\Big]\lambda(t,u)\nu(\dif u)\dif t\no\\
&&+\int_{\mU_0}\left[v(t,X_{t-}+f(t,X_{t-},u))-v(t,X_{t-})\right]\tilde{N_\lambda}(\dif t, \dif u)\no\\
&&+\<(\sigma^*\nabla v)(t,X_t),\dif W_t\>_{\tilde{\mH}}.
\label{itofor}
\ee
Combining (\ref{chco1})(\ref{chco2})(\ref{chco3}) with (\ref{itofor}),
one could have
\ce
&& \dif v(t, X_t)\\
&=&\left[\frac{1}{2}\left\|\sigma^{-1}(t,X_t)b(t,X_t)\right\|_{\mH}^2+\int_{\mU_0}\Big(\big(\log\lambda(t,u)\big)\lambda(t,u)+\big(1-\lambda(t,u)\big)\Big)\nu(\dif u)\right]\dif t\no\\
&&\quad +\int_{\mU_0}\log\lambda(t,u)\tilde{N}_{\lambda}(\dif t, \dif u)+\<\sigma^{-1}(t,X_t)b(t,X_t),\dif W_t\>_{\tilde{\mH}}\\
&=&\<\sigma^{-1}(t,X_t)b(t,X_t),\dif W_t\>_{\tilde{\mH}}+\frac{1}{2}
\left\|\sigma^{-1}(t,X_t)b(t,X_t)\right\|_{\mH}^2\dif t\no\\
&&\quad +\int_{\mU_0}\log\lambda(t,u)N_{\lambda}(\dif t, \dif u)
+\int_{\mU_0}(1-\lambda(t,u))\nu(\dif u)\dif t.
\de
Integrating the above equality from $0$ to $t\in[0,T]$, we know that
\ce
v(t, X_t)-v(t, x_0)=Y_t=-\log\Lambda_t.
\de
By simple calculation, that is exactly (\ref{pathin}).

Next, we prove necessity. On one side, there exists a $C^{1,2}$-function $v(t, x)$ such that $v(t, X_t)$ satisfies (\ref{pathin}), i.e.
\be
\dif v(t, X_t)&=&-\dif \log\Lambda_t=\dif Y_t=\Big[\frac{1}{2}\left\|\sigma^{-1}(t,X_t)b(t,X_t)\right\|_{\mH}^2\no\\
&&+\int_{\mU_0}\Big(\big(\log\lambda(t,u)\big)\lambda(t,u)+\big(1-\lambda(t,u)\big)\Big)\nu(\dif u)\Big]\dif t\no\\
&&+\int_{\mU_0}\log\lambda(t,u)\tilde{N}_{\lambda}(\dif t, \dif u)+\<\sigma^{-1}(t,X_t)b(t,X_t),\dif W_t\>_{\tilde{\mH}}.
\label{pathin1}
\ee
Moreover, based on (\ref{pathin1}) we conclude that $v(t, X_t)$ is a c\`adl\`ag semimartingale with
a predictable finite variation part. On the other side, note that $X_t$ solves Eq.(\ref{3}) and $v(t, x)$ is a $C^{1,2}$-function. Applying
Proposition \ref{itoform} to the composition process $v(t, X_t)$, one could obtain (\ref{itofor}), i.e.
\ce
\dif v(t, X_t)&=&\frac{\partial}{\partial t}v(t,X_t)\dif t+\<\cA X_t,\nabla v(t, X_t)\>_{\mH}\dif t \no\\
&&+\<b(t, X_t),\nabla v(t, X_t)\>_{\tilde{\mH}}\dif t
+\frac{1}{2}[Tr(\sigma\sigma^*)\nabla^2 v](t, X_t)\dif t\no\\
&&+\int_{\mU_0}\Big[v(t,X_{t-}+f(t,X_{t-},u))-v(t,X_{t-})\no\\
&&\qquad -\<f(t,X_{t-},u),\nabla v(t,X_{t-})\>_{\tilde{\mH}}\Big]\lambda(t,u)\nu(\dif u)\dif t\no\\
&&+\int_{\mU_0}\left[v(t,X_{t-}+f(t,X_{t-},u))-v(t,X_{t-})\right]\tilde{N_\lambda}(\dif t, \dif u)\no\\
&&+\<(\sigma^*\nabla v)(t,X_t),\dif W_t\>_{\tilde{\mH}}.
\de
Thus, the above equality is another decomposition of the semimartingale $v(t, X_t)$. By uniqueness for decomposition
of the semimartingale (\cite{dm}), it holds that for $t\in[0,T]$,
\ce
\sigma^{-1}(t,X_t)b(t,X_t)&=&\sigma(t,X_t)^*\nabla v(t,X_t),\\
\log\lambda(t,u)&=&v(t,X_{t-}+f(t,X_{t-},u))-v(t,X_{t-}), \quad u\in\mU_0,
\de
and
\ce
&&\frac{1}{2}\left\|\sigma^{-1}(t,X_t)b(t,X_t)\right\|_{\mH}^2+\int_{\mU_0}\Big(\big(\log\lambda(t,u)\big)\lambda(t,u)+\big(1-\lambda(t,u)\big)\Big)\nu(\dif u)\\
&=&\frac{\partial}{\partial t}v(t,X_t)+\<\cA X_t,\nabla v(t, X_t)\>_{\mH}+\<b(t, X_t),\nabla v(t, X_t)\>_{\tilde{\mH}}
+\frac{1}{2}[Tr(\sigma\sigma^*)\nabla^2 v](t, X_t)\no\\
&&+\int_{\mU_0}\Big[v(t,X_{t-}+f(t,X_{t-},u))-v(t,X_{t-})
-\<f(t,X_{t-},u),\nabla v(t,X_{t-})\>_{\tilde{\mH}}\Big]\lambda(t,u)\nu(\dif u), \quad a.s..
\de
Based on Lemma \ref{supthe} and our assumptions on $\cA\nabla v(t,x)$, we have that
\be
\sigma^{-1}(t,x)b(t,x)&=&\sigma(t,x)^*\nabla v(t,x), \qquad\qquad\qquad\quad  (t,x)\in[0,T]\times\mH, \label{chco11}\\
\log\lambda(t,u)&=&v(t,x+f(t,x,u))-v(t,x), \quad (t,x,u)\in[0,T]\times\mH\times\mU_0, \label{chco22}
\ee
and
\be
&&\frac{1}{2}\left\|\sigma^{-1}(t,x)b(t,x)\right\|_{\mH}^2+\int_{\mU_0}\Big(\big(\log\lambda(t,u)\big)\lambda(t,u)+\big(1-\lambda(t,u)\big)\Big)\nu(\dif u)\no\\
&=&\frac{\partial}{\partial t}v(t,x)+\<\cA x,\nabla v(t,x)\>_{\mH}+\<b(t,x),\nabla v(t,x)\>_{\tilde{\mH}}+\frac{1}{2}[Tr(\sigma\sigma^*)\nabla^2 v](t,x)\no\\
&&+\int_{\mU_0}\Big[v(t,x+f(t,x,u))-v(t,x)
-\<f(t,x,u),\nabla v(t,x)\>_{\tilde{\mH}}\Big]\lambda(t,u)\nu(\dif u).
\label{chco33}
\ee
By simple computation, (\ref{chco11})(\ref{chco22}) correspond to (\ref{chco1})(\ref{chco2}), respectively.
Moreover, (\ref{chco11})(\ref{chco22}) together with (\ref{chco33}) yield to (\ref{chco3}). The proof is completed.
\end{proof}

The above theorem gives a necessary and sufficient
condition, and hence a characterization of path-independence for the density $\Lambda_t$ of the
Girsanov transformation for a SEE with jumps in terms of a IDE. Namely,
we establish a bridge from Eq.(\ref{3}) to a IDE.

\br\label{f0}
Let $f(t,x,u)=0$, and then Eq.(\ref{3}) has no jumps. In Theorem \ref{chth1}, based on (\ref{chco2}),
we know that $\lambda(t,u)=1$ for $u\in\mU_0$. Thus, Eq.(\ref{pathin}) becomes
\ce
\Lambda_t&=&\exp\{-\int_0^t\<\sigma^{-1}(s,X_s)b(s,X_s),\dif W_s\>_{\tilde{\mH}}-\frac{1}{2}\int_0^t
\left\|\sigma^{-1}(s,X_s)b(s,X_s)\right\|_{\mH}^2\dif s\}\\
&=&\exp\{v(0,x_0)-v(t, X_t)\}.
\de
By Theorem \ref{chth1}, the above equation holds if and only if
\ce
b(t,x)&=&(\sigma\sigma^*\nabla v)(t, x), \qquad\qquad\qquad\qquad\quad  (t,x)\in[0,T]\times\mH,\\
\frac{\partial}{\partial t}v(t,x)&=&-\frac{1}{2}[Tr(\sigma\sigma^*)\nabla^2 v](t,x)-\frac{1}{2}\|\sigma(t,x)^*\nabla v(t,x)\|_{\mH}^2
-\langle x,\cA\nabla v(t,x)\rangle_{\mH}.
\de
This is exactly Theorem 3.1 in \cite{wawu}. That is, our result is more general.
\er

\br
Let $f(t,x,u)=0$, and then Eq.(\ref{chco3}) becomes the following time-reversed partial differential equation,
\ce
\frac{\partial}{\partial t}v(t,x)
=-\frac{1}{2}[Tr(\sigma\sigma^*)\nabla^2 v](t,x)-\frac{1}{2}\|\sigma(t,x)^*\nabla v(t,x)\|_{\mH}^2
-\langle x,\cA\nabla v(t,x)\rangle_{\mH}.
\de
The above type of equations has been analyzed in \cite{wawu}. Moreover, in some special cases (cf. Example 3.1 and 3.2 in \cite{wawu}) it is an infinite-dimensional analogy of the Burgers-KPZ equation that is well known in statistics physics. If $f(t,x,u)\neq0$, the special kind of Eq.(\ref{chco3}) appears in \cite{covo}. There its classical solution and viscosity solution are defined and studied. Furthermore, it is worthwhile to mention that a family of option prices is its viscosity solution.
\er

\medskip

Next, we consider Eq.(\ref{3}) with $\sigma(t,x)=0$, i.e.
\be\left\{\begin{array}{l}
\dif \bar{X}_t=(\cA\bar{X}_t+b(t,\bar{X}_t))\dif t+\int_{\mU_0}f(t,\bar{X}_{t-},u)\tilde{N_\lambda}(\dif t, \dif u), \quad t\in[0,T],\\
\bar{X}_0=x_0,
\end{array}
\right.
\label{Eq2}
\ee
where $b:[0,\infty)\times \mH\rightarrow \mH$ and $f:[0,\infty)\times\mH\times\mU_0\mapsto\mH$ are two Borel measurable mappings. 
Since Eq.(\ref{Eq2}) is driven by a purely jump process, some conclusions about it
will be different from that about Eq.(\ref{3}). Let us describe them in details. By Theorem \ref{exun}, Eq.(\ref{Eq2})
has a unique mild solution denoted by $\bar{X}_t$. Assume:
\begin{enumerate}[{\bf (H4.3)}]
\item
\ce
\exp\Big\{\int_0^T\int_{\mU_0}\left(\frac{1-\lambda(s,u)}{\lambda(s,u)}\right)^2\lambda(s,u)\nu(\dif u)\dif s\Big\}
<\infty.
\de
\end{enumerate}

Set
\ce
\bar{\Lambda}_t:=\exp\bigg\{-\int_0^t\int_{\mU_0}\log\lambda(s,u)N_{\lambda}(\dif s, \dif u)
-\int_0^t\int_{\mU_0}(1-\lambda(s,u))\nu(\dif u)\dif s\bigg\},
\de
and then by similar deduction to the above, $\bar{\Lambda}_t$ is an exponential martingale. Define
a measure $\bar{\mP}_t$ via
$$
\frac{\dif \bar{\mP}_t}{\dif \mP}=\bar{\Lambda}_t.
$$
Under $\bar{\mP}_t$, by Theorem \ref{tgir}, the system (\ref{Eq2}) is transformed as
\ce
\dif \bar{X}_t=(\cA\bar{X}_t+b(t,\bar{X}_t))\dif t+\int_{\mU_0}f(t,\bar{X}_{t-},u)\tilde{N}(\dif t, \dif u).
\de
Note that the drift term still exists.

Now, we study path-independence of $\bar{\Lambda}_t$. By the similar proof to that in Theorem \ref{chth1},
we obtain the following result.

\bt\label{chth2}
Assume {\bf (H3.1)} {\bf (H3.3')} {\bf (H3.4)} and {\bf (H4.3)}. Let $\bar{v}: [0,T]\times\mH\rightarrow\mR$ be a scalar function which is $C^1$
with respect to the first variable and $C^2$ with respect to the second variable such that $[\nabla \bar{v}(t,x)]\in\cD(\cA)$ for any $(t,x)\in[0,T]\times \mH$ and $\|\cA\nabla \bar{v}(t,\cdot)\|_{\mH}$ is bounded locally and uniformly in $t\in[0,T]$. Then the Girsanov density $\bar{\Lambda}_t$ for Eq.\eqref{Eq2}
has the following path-indenpendent property:
\ce
\bar{\Lambda}_t=\exp\{\bar{v}(0,\bar{x}_0)-\bar{v}(t, \bar{X}_t)\}, \quad t\in[0,T],
\de
if and only if
\ce
\lambda(t,u)=\exp\{\bar{v}(t,x+f(t,x,u))-\bar{v}(t,x)\}, \quad (t,x,u)\in[0,T]\times\mR^d\times\mU_0,
\de
and $\bar{v}$ satisfies the following time-reversed equation,
\ce
\frac{\partial}{\partial t}\bar{v}(t,x)&=&-\<\cA x+b(t,x),\nabla\bar{v}(t,x)\>_{\mH}-\int_{\mU_0}\Big[e^{\bar{v}(t,x+f(t,x,u))-\bar{v}(t,x)}-1\no\\
&&\qquad\qquad -\<f(t,x,u),\nabla\bar{v}(t,x)\>_{\mH}e^{\bar{v}(t,x+f(t,x,u))-\bar{v}(t,x)}\Big]\nu(\dif u).
\de
\et


\begin{thebibliography}{999}

\bibitem{awz} S. Albeverioa, J.-L. Wu and T.-S. Zhang: Parabolic SPDEs driven by Poisson white noise. {\it Stochastic
Processes and their Applications,} 74(1998)21-36.

\bibitem{bty} J. Bao, A. Truman and C. Yuan: Stability in distribution of mild solutions to
stochastic partial differential delay equations with jumps. {\it Proc. R. Soc. Lond. Ser. A Math. Phys.
Eng. Sci.}, 465(2009)2111-2134.

\bibitem{covo} R. Cont and E. Voltchkova: Integro-differential equations for option prices in exponential L\`evy models. {\it Finance Stochast.} 9(2005)299-325. 


\bibitem{dpz} G. Da Prato and J. Zabczyk: Stochastic Equations in Infinite Dimensions. Encyclopedia of Mathematics
and its Applications. Cambriddge: Cambridge University Press, 1992.

\bibitem{dm} C. Dellacherie and P. A. Meyer: {\it Probabilities and Potential B: Theory of Martingales}. North-Holland, Amsterdam/New York/Oxford, 1982.

\bibitem{iw} N. Ikeda and S. Watanabe: \emph{Stochastic Differential Equations and Diffusion
Processes}, 2nd ed., North-Holland/Kodanska, Amsterdam/Tokyo, 1989.

\bibitem{jjas} J. Jacod and  A.N. Shiryaev: {\it Limit Theorems for Stochastic Processes}. Springer-Verlag, Berlin, 1987.

\bibitem{mpr} C. Marinelli, C. Pr\'ev\^ot, M. R\"ockner: Regular dependence on initial data for stochastic evolution equations with multiplicative Poisson noise. {\it Journal of Functional Analysis}, 258(2010)616-649.

\bibitem{mm} M. Metivier: Semimartingales: {\it A Course on Stochastic Processes}. De
Gruyer, Berlin, 1982.

\bibitem{pa} K. R. Parthasarathy: {\it Probability measures on metric spaces}. AMS Chelsea Publishing, 2005.

\bibitem{pz} S. Peszat, J. Zabczyk: {\it Stochastic Partial Differential Equations with L\'evy Noise: An Evolution Equation Approach}. Cambridge University Press, 2007.

\bibitem{ppks} P. E. Protter and K. Shimbo: No arbitrage and general semimartingales. {\it Markov Processes
and related Topics: A Festschrift for Thomas G. Kurtz}, 4(2008)267-283.

\bibitem{qiao} H. J. Qiao: Exponential ergodicity for SDEs with jumps and non-Lipschitz coefficients,
{\it J. Theor. Probab.}, 27(2014)137-152.

\bibitem{qw} H. J. Qiao and J.-L. Wu: Characterising the path-independence of the Girsanov transformation for
non-Lipschnitz SDEs with jumps, {\it Statistics and Probability Letters,} 119(2016)326-333.

\bibitem{twwy} A. Truman, F.-Y. Wang, J.-L. Wu, W. Yang: A link of stochastic differential
equations to nonlinear parabolic equations, {\it SCIENCE CHINA Mathematics}, 55(2012)1971-1976.

\bibitem{wawu} M. Wang and J.-L. Wu: Necessary and sufficient conditions for path-independence of
Girsanov transformation for infinite-dimensional stochastic evolution equations, {\it Front. Math. China}, 9(2014)601-622.

\bibitem{FYWang} F.-Y. Wang: Harnack Inequalities for Stochastic Partial Differential Equations.
Springer Briefs in Mathematics. New York: Springer, 2013.

\bibitem{wy} J.-L. Wu and W. Yang: On stochastic differential equations and a generalised Burgers equation.
In {\it Stochastic Analysis and Its Applications to Finance- Festschrift in Honor of Prof. Jia-An Yan (eds T. S. Zhang, X. Y. Zhou)},
Interdisciplinary Mathematical Sciences, Vol. 13, World Scientific, Singapore, 2012, 425-435.
\end{thebibliography}
\end{document}